\documentclass{article} 

\usepackage[greek,english]{babel}  
\usepackage{graphicx}
\usepackage[pdftex,bookmarks]{hyperref}
\usepackage{amsmath}
 \usepackage{amssymb}
\usepackage{amsmath,amsthm}
\usepackage{verbatim}

\def\N{\mathbb N}
\def\Z{\mathbb Z}

\def\R{\mathbb R}

\newtheorem{theorem}{Theorem}[section]
\newtheorem{lemma}[theorem]{Lemma}  

\newtheorem{proposition}[theorem]{Proposition}

\newtheorem{corollary}[theorem]{Corollary} 
\newtheorem{remark}[theorem]{Remark} 
\numberwithin{equation}{section}

\titlepage
\title{On the instability tongues of  the Hill equation\\  coupled with a conservative nonlinear oscillator}
\author {Clelia  Marchionna, Stefano Panizzi}
 \date{Dipartimento di Matematica del Politecnico, Piazza Leonardo da Vinci 32, 20133
Milano, ITALY\\ Dipartimento di Scienze
Matematiche, Fisiche  e Informatiche, Parco Area delle Scienze 53/A, 43126 Parma, ITALY}

\begin{document}

\maketitle

\begin{abstract}  We study the asymptotics for the lengths $L_N(q)$  of the  instability tongues of  Hill equations that  arise  as iso-energetic linearization of   two coupled oscillators   around  a single-mode periodic orbit. We show that for small energies,  {\em i.e.} $q\rightarrow 0$, 
the instability tongues have the same behavior that occurs  in the case of  the Mathieu equation: $L_N(q) = O(q^N)$. 
The result  follows from a theorem which fully characterizes the class of Hill equations   with the same asymptotic behavior. 
In addition, in some significant cases we characterize the shape of the  instability tongues for small energies. 
Motivation  of the paper stems from
recent mathematical works on the theory of suspension bridges.
 \end{abstract}

\noindent {\em Keywords}: Hill equation, Mathieu equation, instability tongues, coupled oscillators, coexistence
\

\noindent {\em Mathematics Subject Classification}:   Primary:    34B30; Secondary: 37C75, 34C15 \vspace{0.2cm} \\

\section{Introduction}
We consider a class of  parameterized Hill equations of the following type,
\begin{equation}
z''(t)      +     (\beta + g(u(t,q))) z(t) =0,  \label{H2}
\end{equation}
in which   $\beta$ represents  the spectral  parameter, and the periodic coefficient  depends (through the real analytic  function g)    on the solution $u=u(t,q)$ of an  initial-value problem for a nonlinear conservative second order  differential equation,
 \begin{equation}
u''(t)  +   4 u(t) +  f(u(t))=0 ,  \qquad u(0)=q, \quad u'(0) =0.    \label{duff1}
\end{equation}

In \ref{duff1}, 
 $q$ is a real parameter, and the function $f$ is  assumed to be real analytic in a neighborhood of $0$, with  $f(x) = O(x^2)$, $x\rightarrow 0$.     Under this assumption, if $q$  is sufficiently small,
 the  solution $u(t,q)$  is periodic with  period $T(q)$. We shall  refer to the period of the Hill equation \ref{H2} as $T(q)$, although in some cases  
 the fundamental period of $g(u(t,q))$ could be a fraction of  $T(q)$.\footnote{If  $f$ and $g$ are odd and even functions respectively, the period of $g(u(t,q))$ is indeed $T(q)/2$. It is not possible to exclude lower periods for exceptional values of $q$. } 
 
We are  interested in certain asymptotic properties of the   instability region of equation \ref{H2}, which is 
 the set of pairs of parameters $(q,\beta)$ such that all solutions of \ref{H2} are unbounded. 
According to the basic  theory of the   Hill  equation  \cite{MW}[ch. II, Th. 2.1],  \cite{E}, [ch. 2, Th. 2.3.1] for  any admissible fixed value of $q$,
the instability  set  in the $\beta$-axis  is the union of an unbounded interval $(-\infty, \beta^+_0(q))$ with a countable family of, possibly empty,
  open intervals  $I_N$, $N=1,2,\dots$, whose  endpoints  $\beta^{\pm}_N(q)$    are the   $T(q)$-periodic eigenvalues  
for  even $N$, or the $T(q)$-anti-periodic eigenvalues for odd $N$.  
When  $\beta$ lies  in the interior of the complementary set  all solutions are bounded.    As functions of $q$, the curves $\beta = \beta^{\pm}_N(q)$   form in the plane $(q, \beta)$ the boundaries of the so-called \emph{instability tongues} (resonance tongues, Arnold's tongues)  of the  Hill  equation. These tongues stem and bifurcate from
a sequence of points on the $\beta$-axis corresponding to the double eigenvalues $\beta^+_N(0)= \beta^-_N(0)=N^2$.
Our main concern is  the asymptotic behavior of $\beta^{\pm}_N(q)$   as $q\rightarrow 0$.
We consider  two types of problems:
\begin{itemize}
  \item[(I)]  The order of tangency of $\beta^{\pm}_N(q)$   as $q\rightarrow 0$, that is the decay rate to zero of the signed  length of the instability tongues  $  L_N(q) =  \beta^+_N(q) - \beta^-_N(q)   $.
  \item[(II)] The shape of the instability tongues for small values  of $q$. We shall distinguish between ``trumpet shaped" tongues, containing a   segment of  the horizontal line $\beta = \beta_N(0)$, and ``horn shaped" ones, whose intersection with the  horizontal line $\beta = \beta_N(0)$ is empty for small $q$  (see Fig. \ref{Fig2} in Section 4). 
\end{itemize}

We postpone motivations  and  results on problem (II) to  Section \ref{additional}.
Problem (I)  is  classical  in the standard theory of the Hill equation with two parameters. For instance, if we  set $f(u)\equiv 0$ in \ref{duff1}
and  $g(u)=u$,  equation  \ref{H2} reduces to the   Mathieu  equation $z'' +(\beta + q \cos(2t)) z=0$, for which the asymptotic length    is known to be
$  L_N(q) = C_N q^N + O(q^{N+1})$,
with precise determination of the coefficient $C_N\neq 0$,   see  \cite{Hale,LK}.   For the standard two-parameters Hill equation,
\begin{equation}
\label{standard}
z'' +(\beta + q \phi(t)) z=0,
\end{equation}
where $\phi$ is  a general $L^2$ and $\pi$-periodic  function,  a  classical result of Erd\'elyi \cite{Er} states that  no better estimate than $L_N(q) = O(q)$ can be expected.  In the case when $\phi(t)$ is a trigonometric polynomial of the form
$$ \phi(t) = \sum_{j=1}^s a_j \cos(2jt),$$
Levy and Keller \cite{LK} (see also  \cite{A} for a different approach)    proved   that the length of the $N$-th resonance interval is at most $C_N q^r$,  where $r$ is the integer part of $N/s$, and presented explicit formulas for $C_N$ when $N$ is a multiple of $s$   (see  also  \cite{H},  and   \cite{Vol} for interesting  extensions to a generalized Ince equation). For the similar, and partly related, problem of the asymptotics   of $L_N$  as $N\rightarrow \infty$,    we refer to    \cite{AS,AD}. 

 In this paper we prove  the following theorem which shows  that, for every  equation \ref{H2}  coupled with \ref{duff1},  the instability tongues    have at least the same order of tangency   of the Mathieu equation, that is $L_N(q) = O(q^N)$   as $q\rightarrow 0$.
\begin{theorem} \label{duffhill}
 Assume that  the functions $f$, $g$   are real analytic  in a  neighborhood of the origin,  with $f(x) = O(x^2)$ as $x\rightarrow 0$. Then, for every  $\; N\in \N$,
there exists a  (possibly vanishing)  constant $C_N$, such that

 \

 \noindent $(A)$  $\qquad \qquad L_N(q) = C_N q^N + O(q^{N+1})\quad $     as $q\rightarrow 0$ .
\end{theorem}

It is not  a simple task   to compute the coefficient $C_N$, but  we shall provide a recursive formula in Appendix A   showing that $C_N$ is a  polynomial  of degree $N$ in  the derivatives of  $f$ and $g$ up to order $N$. We are unable to provide a uniform bound  on the rest $L_N(q) - C_N q^N$ in terms of    $f$, $g$ and $N$.

We stress the fact that $C_N$ is possibly vanishing because the coupled system  \ref{H2}--\ref{duff1}    includes   the classical
Lam\'e equation\footnote{We refer here to the {\em Weierstrassian form} of the Lam\'e equation (see \cite[ch. XV, sect. 15.2]{EMOT} ) :
$$ z'' + (\lambda - m(m+1) \mathcal{P}(t)) z =0,$$
where $\mathcal{P}$  is a suitable translation of a Weierstrass elliptic function.}
 corresponding, in our notations,  to   $f(u) =-  6\, u^2$, and $g(u) = - m(m+1) u$, $m\in \N$.   In this case,  Ince \cite{In} in 1940 showed that    only finitely many, precisely $m$,  instability  intervals (thus tongues) fail to vanish. Equivalently, for all but $2m+1$ eigenvalues, there exist two linearly independent periodic eigenfunctions (coexistence). We shall briefly discuss this subject in Section   \ref{coesistenza}  and Appendix B.

In order to  prove Theorem \ref{duffhill}, we need to  rescale  the time  variable and the spectral parameter  so that  equation \ref{H2} reduces  to a Hill equation whose periodic coefficient has fixed period $\pi$  and  depends analytically on the parameter $q$:
\begin{equation}
 \label{Hill1} z'' + (\lambda +  G(t,q)) z = 0.
 \end{equation}

Once this is done, the theorem is a consequence of the following  characterization of the  periodic coefficients $G(t,q)$   in  \ref{Hill1}  for which the asymptotic relation $(A)$ holds true.
\begin{theorem} \label{hilldirect}
Assume that $G(t,q)$ is an even $\pi$-periodic function, depending analytically on the parameter $q$ in a neighborhood of $0$.
Then the  lengths of the instability tongues of  equation \ref{Hill1} satisfy  the asymptotic estimate $(A)$,  if and only if $\,G(t,q)$ admits the following power expansion,
\begin{equation}
\label{exp}
G(t,q) = \sum_{n=1}^\infty G_n(t) q^n,
 \end{equation}
 in which the time coefficients are  trigonometric polynomials of degree $2n$; that is,
\begin{equation}
\label{exp1}  G_n(t) =\sum_{k=0}^n G_{k,n} \cos(2kt), \qquad G_{k,n} \in \R.
\end{equation}
\end{theorem}

In  Theorem \ref{hilldirect} we emphasize the inverse result that, as far as we know,  is new even in the standard case $G(t,q) =q\phi(t)$, 
when it simply states that if  the  instability tongues of   \ref{standard} satisfy   (A), then either $\phi \equiv 0$ or \ref{standard}
is the Mathieu equation. For this reason we take the liberty of naming  \emph{generalized  Mathieu  equation},  any Hill equation   whose periodic coefficient   admits  an expansion such as \ref{exp}--\ref{exp1}.

A Hill  equation such as \ref{H2} arises quite naturally in physical applications as the variational equation of periodic solutions in Hamiltonian systems with two degrees of freedom.
A typical example is provided by a two-mode conservative system of oscillators  that, for a given regular potential energy function $\Psi$, writes  as follows,
 \begin{align}
\label{H1}
    & u''(t)   +    \frac{\partial}{\partial u} \Psi(u(t),z(t))=0 ,  \\[1 ex]
\label{4}     &z''(t)     +      \frac{\partial}{\partial z} \Psi(u(t),z(t))=0.
\end{align}

If we assume the existence of  a periodic single-mode motion, \emph{i.e.} a periodic  solution  of \ref{H1}--\ref{4} in which one  component, say $u$,  is periodic and the other vanishes, 
 the active mode $u=u(t,q)$ can be seen as parameterized by its initial value $u(0)=q$ in the following way, 
  $$
 u''(t) +   \frac{\partial}{\partial u} \Psi(u(t),0)=0, \qquad u(0) = q, \; u'(0)=0.
 $$ 

The linearization at a fixed energy level (iso-energetic linearization) of the system \ref{H1}--\ref{4} around the periodic orbit
$(u(\cdot,q),0)$ yields  the Hill equation,
 $$
           z''(t)     +      \frac{\partial^2}{\partial z^2} \Psi(u(t),0) \, z(t) =0,
 $$ 
whose analysis, according to Floquet's theory, determines the linearized stability or instability of the single-mode periodic motion. 
Thus the   results in this paper are relevant for   the parametric stability/instability  analysis of the system \ref{H1}--\ref{4} in the case  when  the energy of the coupled oscillators system is small.    Here we consider
 $\beta = \frac{\partial^2}{\partial z^2} \Psi(0,0)$ as a parameter,  $\frac{\partial^2}{\partial u^2} \Psi(0,0)=4$ (possibly after a   suitable rescaling of time),
  $\frac{\partial}{\partial u} \Psi(u(t),0) = 4u + f(u) $,  $ \frac{\partial^2}{\partial z^2} \Psi(u,0) = \beta + g(u)$.


  The main motivation for starting the study of problems (I) and (II) is the analysis of   parametric torsional instability for some recent suspension bridge models, where a  finite dimensional projection of the phase space  reduces the stability analysis at small energies of the model to the stability of a Hill equation such as \ref{H2}. We refer the reader to Gazzola's book  \cite{Gaz},  to the  papers  \cite{BG,BGZ,AG,CAG,F}, and to our previous works  \cite{MP,MP2}. Other interesting applications  arise in  the study of the stability of nonlinear modes in some beam equations \cite{GG}  or string equations \cite{Dic,CW}.   In the latter case, we must observe that the eigenvalue problem takes a different form:  $ \,  z'' + \beta(u+g(u)) z =0 $.
Our results, in particular Theorem \ref{duffhill}, extend to this form as well but in order to avoid redundancy  of quite similar reasonings we do not include the proof.

The plan of the paper is the following:  In Section 2, after introducing the problem in the context of analytic perturbation theory, we prove Theorem \ref{hilldirect}. The direct part is an adaptation of the argument in \cite{LK}, whereas the converse makes use of  a new inductive argument. In Section 3 we deal with our main result    (Theorem \ref{duffhill})  whose proof is, after rescaling, merely a verification of the assumptions of Theorem \ref{hilldirect}; in addition to a few complementary results   we briefly recall the issue of the existence of finitely many tongues (coexistence).   In Section 4 we discuss the shape of the instability tongues depending on  the first coefficients in the expansions of $f$ and $g$. Some examples 
 that are relevant to the theory of suspended bridges are examined in Section 5, and some situations are shown in which   only finitely many    tongues do not vanish; some are well-known while others are novel. 
 
We include two appendices: Appendix A describes a recursive formula for the computation of $C_N$; Appendix B elaborates on a few transformations of the Lam\'e equation  relevant for this work.

\section{The generalized Mathieu Equation}
\label{Mathieu} 

In the first part of this  section we consider the Hill equation \ref{Hill1}, and    the \emph{if} part of Theorem \ref{hilldirect}. The inverse result
  will be proved in the second part of this section. The proof of the direct result is a variation and a simplification  of an argument in  \cite{LK}.
The inverse proof uses a new, although simple, inductive procedure. Before proceeding with the proofs, we point out some general issues on the analytic perturbation problem
we are addressing.

The periodic eigenvalue problem for the Hill equation \ref{Hill1} is   a regular perturbation problem and may be cast in Kato's abstract
framework \cite{K}. We assume that $G(\cdot,q)$ is $\pi$-periodic as a function  of $t$, and  is analytic in a neighborhood of $q=0$  as a function of $q$, with values in $L^\infty([0,\pi])$, i.e.
\begin{equation}
\label{generalG}  G(t,q) = \sum_{n=1}^\infty q^n G_n(t), \qquad \limsup_{n\rightarrow \infty} \|G_n\|_\infty^{1/n}< \infty.
\end{equation} %

To avoid distinction among  periodic (even eigenvalue numbers) and anti-periodic (odd eigenvalue numbers) eigenfunctions, we assume as reference space the Hilbert space $H=L^2([-\pi,\pi])$, in which we consider the family of self-adjoint operators with discrete spectrum,
$$ A(q) = -\frac{{\rm d}^2}{{\rm d} t^2} - \sum_{n=1}^\infty q^n G_n,$$
with boundary conditions $z(-\pi)= z(\pi)$, $z'(-\pi) = z'(\pi)$.
The Hilbert space $H$ may be  decomposed according to
$ H = H^+ \oplus H^- $, where $H^{\pm}$ denotes the subspace of even ($+$) functions, and odd ($-$) functions, that is
$$ H^+ = \mathop{\rm span}\{\cos kt : \, k\geq 0 \}, \quad H^- = \mathop{\rm span}\{ \sin kt: \, k\geq 1\}. $$

Consequently, with obvious notation,  we have $A(q) = A(q)^+ \oplus A(q)^- $, so that
 the doubly degenerate eigenvalues $\lambda_N(0) =N^2$ turn out to be simple in $H^{\pm}$. Owing to the Rellich--Kato perturbation theorem (see e.g. \cite{V}), every perturbed  eigenvalue $ \lambda^{\pm}_N (q)$ in $H^{\pm}$ depends analytically on $q$.  We shall write the power series
\begin{equation}
\label{eigexp} \lambda^{\pm}_N (q) = N^2 + \sum_{n=1}^\infty \Lambda_n^{\pm}(N) \,q^n,
\end{equation}
whose convergence radius $r_N$ can be estimated by Kato's  resolvent method: a lower bound for $r_N$ is given by the solution of the following equation (see \cite[ch. II, \S 3]{K}),
$$ \sum_{n=1}^\infty r_N^n  \|G_n\|_\infty = d_N/2, $$
where $d_N$ is the isolation distance\footnote{The isolation distance is the distance of $\lambda_N$ from the the rest of the spectrum. It can be raised by the additional decomposition of $H$ into  periodic and anti-periodic functions, see \cite[ch. VII, \S 3]{K}.   } of $\lambda_N(0)= \lambda_N^{\pm}(0)$, i.e. $d_N = N^2 -(N-1)^2$.

From now on  in this section,   to avoid proliferation of indices, we  omit
the dependence on the eigenvalue number $N$, which we consider as fixed.   We denote by $Z^{\pm}(t,q)$ the  even ($+$) and odd  ($-$) normalized (see below \ref{constraint}) eigenfunction  corresponding to $\lambda^{\pm}_N$,  whose  power series
expansion  is  given by
\begin{equation}
\label{ps} Z^{\pm}(t,q) = \sum_{n=0}^\infty q^n z^{\pm}_n(t).
\end{equation}

If we plug   the power series  expansions \ref{eigexp}, \ref{ps}   into the equation \ref{Hill1}, we get the following recursive sequence of differential equations,
\begin{eqnarray}  \label{z0} && z_0^{\prime \prime}  + N^2 z_0  =0,\\
\label{zn}  && z_n^{\prime \prime} + N^2  z_n+   \sum_{s=1}^n \Lambda_s z_{n-s}+ \sum_{s=1}^n G_s(t) z_{n-s} =0 \qquad n\geq 1.
\end{eqnarray}

The $2\pi$-periodic solutions to \ref{z0}--\ref{zn} are not  unique, unless   we assume an   additional constraint,  such as the following,
\begin{equation}
\label{constraint}
\frac{1}{2\pi} \int_{-\pi}^\pi Z^{+}(t,q) \cos(Nt)  {\rm d}t = \frac{1}{2\pi i} \int_{-\pi}^\pi Z^{-}(t,q) \sin(Nt)  {\rm d}t =1.
\end{equation}

\subsection{Proof of Theorem \ref{hilldirect}:  Direct problem}  

Here we assume that all  coefficients $G_n$ are even $\pi$-periodic trigonometric polynomials of degree $2n$ such as in \ref{exp1}, and prove the property (A).   The proof is divided into two steps: first we consider the Fourier expansion
of
each $z^{\pm}_n$, and write down recursive formulas for $\Lambda_n^\pm$,  $z^{\pm}_n$; the rest of the proof relies mainly  on a finite propagation speed of disturbances property of the  system \ref{ric}--\ref{Bj},  which can be expressed either by the law of enlargement of supports or by the dual concept of domain of dependence, and is contained in three Lemmas; the last one, Lemma \ref{L3}, shows  that for $N\geq 1$ the order of tangency of $\lambda^{\pm}_N(q)$ at $q=0$ is at least $N-1$, that is $\Lambda_n^+(N)= \Lambda_n^-(N)$ in the expansion \ref{eigexp}, for $n\leq N-1$. Of course this is equivalent to the asymptotic estimate (A) with $C_N=  \Lambda_N^+(N)-\Lambda_N^-(N) $.

The  Fourier expansion of
each $z^{\pm}_n$ is:
\begin{equation}
\label{fourier}  z_n^{\pm}(t) = \sum_{k=-\infty}^\infty z_{k,n}^{\pm}e^{ikt}, \qquad z_{-k,n}^{\pm} = \pm z_{k,n}^{\pm},
\end{equation}
 where the first component of  the pair of indices
$(k,n) \in \Z\times\N$    refers to frequency, the latter  to the power of $q$,
We note that, owing to   \ref{ps}, \ref{constraint}, and \ref{fourier},  we get the initial conditions at level $n=0$,
\begin{equation}
\label{incond}  z_{ k,0}^{\pm} = \delta_{k,N} \pm \delta_{k,-N},
\end{equation}
and the fact that the $N$-th Fourier coefficient   of $z_n$ is zero for $n\geq1$, that is
 \begin{equation}  \label{keqN}
 z_{\pm N,n}^{\pm} =0, \qquad n\geq 1.
 \end{equation}

By substituting \ref{fourier} in \ref{zn}, we obtain the following recursive system for $z^{\pm}_{k,n}$, and $\Lambda_n^{\pm}$ \footnote{The same tecnique applies also for $N=0$, in order to compute $\lambda^+_0 (q) = \sum_{n=1}^\infty \Lambda_n^+(0) \,q^n$,  the upper bound of the 0-th unbounded interval of instability. The formulas \ref{ric},  \ref{Bj} are also true, providing to start with $z^+_{k,0}=\delta_{k,0}$, accordingly to \ref{constraint}.},
\begin{eqnarray}
\label{ric}
(N^2 - k^2) z_{k,n} &=&  - \frac{1}{2} \sum_{s=1}^n \sum_{i=0}^s G_{i,s} ( z_{k-2i,n-s} + z_{k+2i,n-s} ) - \sum_{s=1}^n \Lambda_s z_{k,n-s} ,\\
\label{Bj} \Lambda_n &=&  - \frac{1}{2} \sum_{s=1}^n \sum_{i=0}^s G_{i,s} \left( z_{N-2i,n-s} + z_{N+2i,n-s} \right).
\end{eqnarray}

The  second equation \ref{Bj} is obtained   either by taking the scalar product  of \ref{zn} with $e^{iNt}$ or    by setting $k=N$ in \ref{ric}. We note that  the symmetry relations $z^{\pm}_{k,n}= \pm z_{k,n}^{\pm}$ are satisfied, since the system \ref{ric}--\ref{Bj} is invariant under the transformation $k \mapsto -k$, and in the same way, one could get  an equation equivalent to \ref{Bj} by setting
$k=-N$ in \ref{ric}.

As in \cite{LK}, we  need  the following lemmas on the vanishing coefficients of system \ref{ric}--\ref{Bj}.
\begin{lemma}
\label{L1}
The frequency index $k$  of non vanishing coefficients must have  the same parity of $N$, that is
$  z_{k,n}^{\pm} =0$ for odd  $ k-N $.
The indices of non vanishing coefficients are contained in the union of  two forward cones:
$$ S_N =\{(k,n) \in \Z\times\N: \, |k-N| \leq 2n\} \cup \{(k,n) \in \Z\times\N: \,  |k+N| \leq 2n\}. $$
that is $z_{k,n}^{\pm} =0$, if $(k,n)$ belongs to the complementary set of $S_N$.

\end{lemma}

\begin{proof} The assertion on the parity of $k-N$ is easily proved by induction, but it is obvious if we think that for
even/odd $N$, $z_n^{\pm}$  is a periodic/anti-periodic function.  The other assertion  is proved  by induction on $n$.
  For  $n=0$ the assertion is true by the initial conditions \ref{incond}. Assume that  it is true up to the level $n-1$, that is
$  z_{h,m}^{\pm} = 0$,  if $(h,m) \notin S_N$, and $m\leq n-1$. We remark that, for a given pair of indices  $(k,n)\in \Z \times \N$, all  the indices of  $z_{k-2i,n-s}$, $z_{k+2i,n-s}$, $z_{k,n-s}$ in formula
\ref{ric} belong to the following backward cone:
\begin{equation}
\label{Bcone}
C_{\mathbf{k,n}} = \{ (h,j)\in \Z \times \N:  |k-h| \leq 2(n-j)\}\setminus\{(k,n)\}.
\end{equation}

By a  simple but cumbersome check, we have that if the vertex $(k,n)$ of $C_{\mathbf{k,n}} $ does not belong to
$S_N$, then $ C_{\mathbf{k,n}}\cap S_N = \emptyset$, and $j<n$ if  $(h,j)\in C_{\mathbf{k,n}}$.
Thus we get   $z_{k,n}=0$, if   $(k,n)\notin  S_N$.

\end{proof}


\begin{lemma}
\label{L2} The domain of dependence of $\Lambda_{n}^{\pm}$ is the backward cone  $C_{\mathbf{N,n}}$, as defined in \ref{Bcone}.
  The domain of dependence of $z^{\pm}_{k,n}$ is the backward cone $C_{\mathbf{k,n}}$.
  This means that the value of $z^{\pm}_{k,n}$ is not influenced by any  $z_{h,j}^{\pm}$ if  $(h,j) \notin C_{\mathbf{k,n}}.$

\end{lemma}
\begin{proof}  The assertion  on the domain of dependence of $\Lambda_{n}^{\pm}$   is verified by direct  inspection of the indices in \ref{Bj}.
Let us verify the assertion on the cone of $z^{\pm}_{k,n}$. As we noted in the proof of Lemma \ref{L1}, every index
of the $z$'s appearing in \ref{ric} belongs to $C_{\mathbf{k,n}}$. We need to take care of the domains of dependence of
 the terms
$\Lambda_s^{\pm}$, with $s\leq n$, appearing in formula \ref{ric}. We assume for the moment $k \geq 0$. The case $k=N$ is obvious.
If $|k-N| = 2h >0$, we remark that, owing to Lemma \ref{L1}, the summation $  \sum_{s=1}^{n} \Lambda_s^{\pm} z_{k,n-s}^{\pm}$ does  not extended up to $n$. Indeed we have
$z_{k,0}^{\pm} = z_{k,1}^{\pm}  =\dots= z_{k,h-1}^{\pm} =0$, since their indices do not belong to the support set $S_N$, as it seen by the inequality
$|k-N| =2h > 2(n-s)$, $s>n-h$.
Therefore summation can be replaced by (intended to vanish if $h\geq n$),
\begin{equation}
\label{Bs}
 \sum_{s=1}^{n-h} \Lambda_s^{\pm} z_{k,n-s}^{\pm}, \qquad 2 h=|k-N|.
\end{equation} 

Since  $C_{N,s} \subset C_{N,j}$ if $s\leq j$, the largest cone of dependence of the terms $\Lambda_s$ in \ref{Bs} is $C_{\mathbf{N,n-h}}$ corresponding to the largest index $n-h$.
By definition of  $h$, $2h= |N-k| \leq 2|n-(n-h)| = 2h$, thus  its  vertex $(N,n-h)$
 belongs to  $C_{k,n}$. It follows that  the whole cone is contained in $C_{k,n}$.
This proves the assertion on the dependence cone of $z_{k,n}^{\pm}$, if $k\geq0$.
The case  $k<0$ reduces to the previous one by symmetry, since $z_{-k,n}^{\pm} = \pm z_{k,n}^{\pm}$
\end{proof}


\begin{figure}[htbp]\hspace{0.5cm}
\begin{minipage}{5cm}
\begin{center}

\scalebox{0.25}{\includegraphics{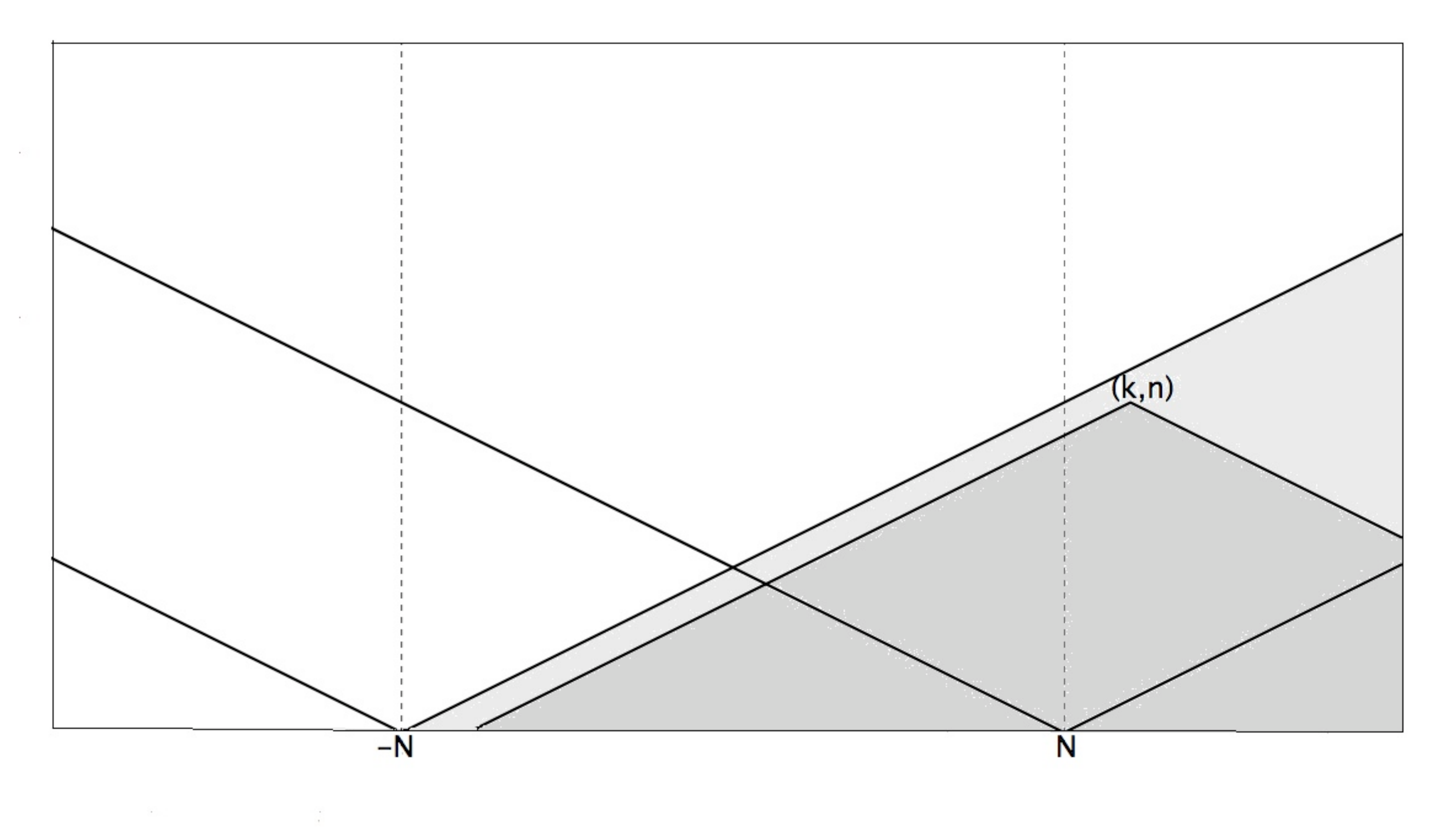}}
\end{center}
\end{minipage}
\caption{The shaded region represents  the set $R$ in which $z^{+}_{h,m} = z^{-}_{h,m} $. The darker region is its intersection with a domain of dependence $C_{\mathbf{k,n}}$, when $k>2n -N$} \label{Fig1}
\end{figure}


The main issue in the proof of Theorem \ref{hilldirect} consists in  identifying  the region in the plane $(k,n)$ in which $z_{k,n}^{+} = z_{k,n}^-$,  this is set out by the following Lemma:

\begin{lemma}
\label{L3}
Let $R$ be the region below the line $k=2n-N$, that is
$$ R = \{(k,n)\in \Z \times \N: \; k>2n -N \}.$$
Then we have $z_{k,n}^+ = z_{k,n}^-$, for every  $(k,n) \in R$, and  consequently  $\Lambda_n^+=\Lambda_n^-$ for $n\leq N-1$.

\end{lemma}

\begin{proof}
Let us set
$$ R_n = \{(k,j)\in R: \quad j\leq n\}.$$
We prove the assertion by induction on $n$.
We have    $z^{+}_{k,j} = z^{-}_{k,j}$ for  $ (k,j) \in R_0$, since the only non vanishing term is   $z^{\pm}_{N,0} =1$.
Assume that $z^{+}_{k,j} = z^{-}_{k,j}$ for every $ (k,j) \in R_{n-1}$.  Since the  domain of dependence of  $z^{\pm}_{k,n}$, with $(k,n)\in R_n$ is contained in  $R_{n-1}$,  we get  $z^{+}_{k,j} = z^{-}_{k,j}$ for every $ (k,j) \in R_n$.

We observe that the domain of dependence $C_{\mathbf{N,n}} $ of $\Lambda_n^{\pm}$ is contained in $R$   if $n\leq N-1$, thus the rest of the assertion follows
by  formula \ref{Bj} and Lemma \ref{L2}.

\end{proof}


\begin{remark} \label{rem2}
 Let     $G(t,q)$  be a function as in the assumptions of Theorem \ref{hilldirect}. If, for some $K>1$, we have
 $G_i\equiv0$, for $i=1,\dots,K-1$, then, in addition to (A), we have
$$
L_N(q)=O(q^K), \quad   N \leq K.
$$
In fact, from  formula 
\ref{Bj} we have immediately that  $\Lambda_i^\pm=0$
 for $i<K$, for $i=1,\dots,K-1$.
\end{remark}

\begin{remark} \label{rem1}
 Let $m\geq 1$ be  a fixed integer, and 
let us weaken the assumption on the $\pi$-periodic coefficients $G_n$   by requiring that  they  are 
polynomials of degree at most $2n$, for $n\leq m$ (instead of  $n\in \N$).  
Then  Lemmas 1, and 2 hold true up to the level $m$.  This means that  in Lemma 1, the domain of dependence of  $z_{k,n}^{\pm}$ is still
$  C_{\mathbf{k,n}}$, provided $n\leq m$,  while in Lemma 2, we have $z_{k,n}^+ = z_{k,n}^-$, for every  $(k,n) \in R$, with $n\leq m$.
It follows that  in  Theorem \ref{hilldirect}, we  still have 
$  L_N(q) = O(q^N)$  for the first $m$ instability tongues. 
\end{remark}

For future reference, we report here the computation of  the two first coefficients $\Lambda_1^\pm$ and $\Lambda_2^\pm$ of $\lambda_N^\pm(q)$ in \ref{eigexp}. By using \ref{incond}, \ref{keqN} and \ref{Bj}, we get the following expressions,
\begin{equation}
\label{B1}
 \Lambda^\pm_1(1) = -G_{0,1} \mp \frac{1}{2} G_{1,1},  \qquad
                              \Lambda^\pm_1(N)  =  -G_{0,1},  \qquad N=0, \quad N\geq 2.
\end{equation}
\begin{eqnarray}
\label{B2}
  \Lambda^\pm_2(1) &=& -G_{0,2} -\frac{1}{32} G_{1,1}^2 \mp \frac{1}{2} G_{1,2}, \\
	\label{B22}	             \Lambda^\pm_2(2) &=& -G_{0,2} +\frac{1}{24} G_{1,1}^2 \pm  \left(- \frac{1}{2} G_{2,2} + \frac{1}{16}G_{1,1}^2 \right), \\
                       \label{B23}       \Lambda^\pm_2(N)&=& -G_{0,2} + \frac{1}{8(N^2-1)}G_{1,1}^2, \qquad N=0, \quad N \geq 3.
\end{eqnarray}


\subsection{ Proof of Theorem \ref{hilldirect}:  Inverse problem} 
Here we  consider the Hill equation \ref{Hill1}
under the general assumption that $G$ is an even $\pi$ periodic function satisfying \ref{generalG} without restrictions
on the degree of $G_n$, and we prove the \emph{only if} part of Theorem \ref{hilldirect}.

We remark that formula \ref{Bj} for the coefficients in  the expansion of the eigenvalues $\lambda^{\pm}_N (q)$ is now replaced
by the the following summation
\begin{equation}
\label{Bjg} \Lambda_n^{\pm} = - \frac{1}{2} \sum_{s=1}^n \sum_{i=0}^\infty G_{i,s} \left( z_{N-2i,n-s}^{\pm} + z_{N+2i,n-s}^{\pm}  \right).
\end{equation}

 First of all, let us prove that under assumption (A), $G_1$ is a polynomial of degree at most $2$.
Let  $N\geq 1$ be an arbitrary eigenvalue number, and let us apply  formula
\ref{Bjg} for $n=1$. We have
$$ \Lambda_1^{\pm} = -\frac{1}{2}   \sum_{i=0}^\infty G_{i,1} \left( z_{N-2i,0}^{\pm} + z_{N+2i,0}^{\pm} \right).$$

Since  $z_{k,0}^{\pm} =\delta_{k,N} \pm \delta_{k,-N}$, we get
$$ \Lambda_1^{\pm} = -  G_{0,1} \mp \frac{1}{2} G_{N,1},$$
thus $ \Lambda_1^+ - \Lambda_1^-  = -G_{N,1}$. We infer  that   $L_N(q) =  - G_{N,1}  q + O(q^2)$ for every  $N\geq 1$.
Owing to the assumption  (A), we conclude that
 $G_{N,1} = 0$ for  $N>1$, which proves the assertion.

 Now let us consider an integer  $m\geq 2$, and  assume that
\begin{equation}
\label{pn} G_n(t) = \sum_{k=0}^n G_{k,n} \cos(2k t), \quad  \text{ for every }  n\leq m-1,
\end{equation}
that is $G_n$ is a polynomial of degree at most $2n$ for $n\leq m-1$.  
We shall show that  \ref{pn} leads to   $L_N = - G_{N,m}  q^m + O(q^{m+1}) $, for every $N>m$. 
Thanks to (A) we conclude that $G_{N,m}=0$, for every   $N>m$,   which means that   $G_m$ is a polynomial
of degree at most $2m$. Thus the assertion  will follow by induction on $m$.

Let us consider  the $N$-th eigenvalue branch $\lambda^\pm_N$, with $N>m$. Under assumption \ref{pn},  Lemma 2  and Lemma 3 hold true  for all levels  $n\leq m-1$ (see Remark \ref{rem1}), in particular  $\Lambda_n^+ = \Lambda_n^-$ for  $n\leq m-1$.
Let us apply  \ref{Bjg} for $n=m$. We have
\begin{equation}
\label{BN} \Lambda_m^{\pm}  =  - \frac{1}{2} \sum_{s=1}^{m-1}\sum_{i=0}^s G_{i,s} \left( z_{N-2i,m-s}^{\pm} + z_{N+2i,m-s}^{\pm}  \right)
   - \frac{1}{2} \sum_{i=0}^\infty G_{i,m} \left( z_{N-2i,0}^{\pm} + z_{N+2i,0}^{\pm}
\right)
\end{equation}

The first term on the right-hand side of  \ref{BN} does not depend on the determinations $\pm$, since all the indices $(N\pm 2i,m-s)$   are in the region $R$, up to the level $m$; let   $A$ be its value.   Thus, by the initial conditions at level $n=0$,   we get
$$  \Lambda_m^{\pm} = A - G_{0,m} \mp \frac{1}{2} G_{N,m}.$$

It follows that   $L_N = - G_{N,m}   q^m + O(q^{m+1}) $, for  every $N>m$. This concludes the proof of  Theorem \ref{hilldirect}.

\subsection{Existence of finitely many tongues}
\label{coesistenza}
 We point out that not only the instability tongues can  be thinner than predicted by the general result, but can even disappear. We will show some examples of existence of finitely many tongues  in   Section  \ref{examples}. 
 
 The question of the existence of  finitely many instability intervals (gaps) for the Hill equation, 
$$
z''(t)+(\beta+ Q(t)) z(t)=0
$$ 
has been deeply investigated by many authors, and dates back to the work of Ince \cite{In22} on the impossibility of the coexistence\footnote{This is the name of the subject in classical literature. Coexistence means the existence of two linearly independent eigenvalues, a condition equivalent to the vanishing of the instability interval.} for the Mathieu equation, see \cite[ch. VII]{MW}, and   \cite{DM} for interesting extensions and a recent account of the subject.  A detailed study of the coexistence problem for the related Ince equation is provided by \cite{Vol1}.

Starting from the introduction of the Lax pairs  formulation of the KdV hierarchy as a compatibility relation with the Hill operator, research on the multiplicity of eigenvalues    has come to a remarkable and celebrated result, essentially  thanks to the work of Lax  \cite{Lax} and Novikov  \cite{Nov}  around 1975  (see also  \cite{gold}):   at most  $n$ instability intervals fail to vanish if and only if   $Q$  satisfies a differential equation of the form, 
\begin{equation}
\label{kdv} Q^{(2n)} + H(Q,Q',\dots,Q^{(2n-2)})=0,
\end{equation}
where $H$ is a polynomial of maximal degree $n+2$. It turns out that equation \ref{kdv} is equivalent to  a linear combination of the first  $n$-order stationary KdV equations. We refer to  \cite{GW}     for an extensive bibliography, and a clear presentation of the modern theory.

In the  starting case  $n=1$,  there exists exactly one finite instability  interval if and only if $Q(t)$ satisfies the  equation $\, Q'' + AQ +B+ 3 Q^2 =0\,$
for  suitable real constants $A$, $B$  (the first proof of the necessity of this condition is due to Hochstadt \cite{H}).

For $n> 1$,  in the rest of the paper  we  will  refer to the following classical result of Ince \cite{In,EMOT} \cite[ch. VII]{MW},  on a particular class of elliptic coefficient of the Hill equation  offering  the simplest example  for which all but $n$  finite instability intervals disappear. Here we state the theorem  in a favorable form  for our purposes, see Appendix B for a brief discussion.
 
 \begin{theorem}[The Ince  theorem]
 \label{incemod}
Let   $Q$ be a non constant periodic solution of the differential equation, 
\begin{equation}
\label{Q}
Q'' + AQ +B+ 3 Q^2 =0
\end{equation}
where $A$, $B$ are real numbers such that $A^2-12B>0$. Then, for every positive integer $n$,  the Hill equation, 
$$
z''(t)+ \left(\beta+ \frac{n(n+1)}{2}\, Q(t)\right) z(t)=0
$$ 
has   exactly      $n+1$ instability intervals, including the unbounded one.
\end{theorem}

In   Section \ref{examples} we will  provide some examples of coupled equations \ref{H2}--\ref{duff1} where  equation  \ref{H2} can be written in the form, 
$$z''(t)+ \left(\beta+ \gamma\, Q(t,q)\right) z(t)=0$$
with $Q(t,q)$ satisfying \ref{Q} for every $q$.
As a consequence, if  $\displaystyle{ \gamma=\frac{n(n+1)}{2}}$,   only a finite number of tongues  fail to vanish.
 
 
\section{Proof of Theorem \ref{duffhill}} \label{proof1}

In this section we  prove Theorem \ref{duffhill}. In the first part we provide the asymptotic development of the periodic solutions of equation \ref{duff1} by removing secular terms as in the classical Poincar\'e--Lindstedt method. In the second part we insert the development in the Hill equation \ref{H2}, and after an adequate normalization of the coefficients, we show that the assumptions of Theorem  \ref{hilldirect} are satisfied.
\subsection{Expansion of the solution of equation \ref{duff1}}
Let $u(t,q)$ be the solution to the initial-values problem \ref{duff1}.
According to our assumptions on the function $f$, we  write the   Taylor series of $f$ in a neighborhood of $0$:
$$ f(x) = \sum_{k=2}^\infty \alpha_k x^k, \qquad |x| <r_0.$$

Let   $r_1$ be the least modulus of the singular points of the equation \ref{duff1}, that is  $r_1 = \min \{ |x| : \, 4x+f(x) =0\}$,  $r_0=+\infty$ in case the set is empty. The  parameter $q$ will be subject to several restrictions, the first one being  $|q| < \min\{r_0,r_1\}$ so that the solution of \ref{duff1} are periodic and depend analytically on $q$. From now on we simply assume that the parameter $q$ is  small enough so that our power series converge.

Let us denote by    $T(q)$  the period of $u(t,q)$ and by $ \omega(q)= \pi/T(q)$  its angular frequency.  Both depend analytically on $q$ in some (in general) smaller neighborhood of $0$, thus we can write the following power series expansion  ($\Omega_0=1$),
\begin{equation}
\label{Omega}
 \Omega(q) = \omega(q)^2 = \sum_{n=0}^\infty q^n \Omega_n.
\end{equation}

If we rescale time in  \ref{duff1}  by setting   $\tau = \omega(q) t$,  and  the solution
$ u(t,q) = q U(\tau,q)$,  so that $U(\tau +\pi;q) = U(\tau,q)$, the problem \ref{duff1}  reads as follows,
\begin{equation}
\label{U}
\Omega(q) U''(\tau) + 4 U(\tau) + \sum_{n=1}^{\infty} \alpha_{n+1} q^n U(\tau)^{n+1} =0, \qquad U(0)= 1, \quad U'(0)=0.
\end{equation}

By the Poincar\'e expansion theorem (see \cite[Th. 9.2]{V}),  $U(\tau,q)$ can be expressed, on the fixed time interval $[0,\pi]$ (thus on $\R$), as  a convergent power series with respect to $q$ in a neighborhood of $0$,  uniformly with respect to $\tau$:
\begin{equation}
\label{expansion}
 u(t,q) = qU(\tau,q) = \sum_{n=1}^\infty q^nu_n(\tau).
 \end{equation}

The coefficients $u_n$ in the expansion \ref{expansion} are periodic and, by the initial conditions in \ref{U}, we obtain that
\begin{equation}
\label{u_n}
 u_n(\tau + \pi) = u_n(\tau), \quad u_1(\tau) = \cos (2\tau), \quad u_n(0) = u'_n(0)=0, \quad n\geq2.
 \end{equation}

If we plug the expansion \ref{expansion} into the problem  \ref{U} we get, in addition to conditions \ref{u_n}, the sequence of recurrent differential equations,
\begin{align}
  & u_1'' + 4u_1 =0,  \label{uno}   \\
&   u_2'' +  4u_2 = -    \Omega_1 u_1''  -  \alpha_2 u_1^2 ,   \label{due} \\
&  u_3'' +  4u_3 = -   \Omega_2 u_1'' -  \Omega_1 u_2'' -   2\alpha_2 u_1 u_2 - \alpha_3 u_1^3, \label{tre} \end{align} 
and in general, for $n>3$, 
 \begin{equation}  u_n'' + 4u_n  =F_n(\tau)  \label{n} \end{equation} 
where 
$$ F_n(\tau) = - \sum_{k=1}^{n-1} \Omega_k u_{n-k}''  - \sum_{k=2}^{n} \alpha_{k} \sum_{i_1 +\dots+ i_{k}=n}u_{i_1} \cdots  u_{i_{k}}
$$

Periodic solutions of the $n$-th recurrent equation are possible  if  secular terms are removed  from the right-hand side of the equation, so that  the coefficient of the resonant term  in $F_n(\tau)$ vanishes. This means that we have to impose  that  $\int_0^\pi F_n(\tau) \cos(2\tau) \, {\rm d}\tau =0$,   which  is the first step  to obtain  the asymptotic expansions of  $\omega(q)$,  and subsequently of
$u(t)$,  by the Poincar\'e--Lindstedt method (see \cite[ch. 10]{V}).

By a simple  inductive argument, we  can show  the following property of the coefficients $u_n$:

\begin{proposition}
\label{prop1}
The coefficients  $u_n(\tau)$, $n\geq 1$  in the power series   \ref{expansion}  are even $\pi$-periodic trigonometrical polynomials of degree $2n$.
\end{proposition}

\begin{proof} 

We prove the assertion by induction on $n\in\N$. It is obviously true for $n=1$, and let us assume it is true for $1\leq j \leq n-1$  ($n\geq 2$).
By a simple computation, it follows  that the multilinear terms in $F_n(\tau)$ of  the $n$-th recursive differential equation, that is
$$ \sum_{i_1+\dots+i_{k} =n} u_{i_1} u_{i_2}\cdots  u_{i_{k}}, $$
and the term $\sum_{k=1}^{n-1} \Omega_k u_{n-k}'' $,
are even $\pi$-periodic polynomials of degree  $\leq 2n$.
Thus, once the resonance has been removed,  the source term $F_n$ in the $n$-th equation  has the following expression,
$$ F_n(\tau)  = \sum_{k=0, k\neq1}^{n} c_k\cos(2k\tau).  $$
Therefore, recalling that $u_n(0)=u_n'(0)=0$,   the  solution of the $n$-th problem, is  given by
$$ u_n(\tau) = \sum_{k=0, k\neq1}^{n} \frac{c_k}{4-4k^2} \cos(2k\tau) - \sum_{k=0, k\neq1}^{n} \frac{c_k}{4-4k^2} \cos(2\tau),$$
which proves the assertion.
\end{proof}


\subsection{Hill Equation}

Here we turn our attention to  the periodic eigenvalues problem for  the Hill equation \ref{H2}.
We need to rewrite the equation  in the form \ref{Hill1}:  we  rescale the time variable,  $\tau = \omega(q)t$, set  $z(t) = Z(\omega(q) t)$. Then, by introducing    the new coefficients,
\begin{equation}
\label{Gbeta} \lambda(q) = \beta(q)/\Omega(q) ,\qquad
   G(\tau,q)= g(qU(\tau,q))/\Omega(q) ,
   \end{equation}
we  get rid of the   $\Omega(q)$  factor by  absorbing it in a modified eigenvalues problem,   so that we obtain
a  Hill equation with fixed period  $\pi$:
\begin{equation}
\label{Zold}
Z''(\tau) +  \left(\lambda(q)   +   G(\tau,q) \right)Z(\tau) =0.
\end{equation}

\begin{lemma}
\label{L4}
Let $g$ be a real  analytical function in a neighborhood of $0$, $g(0)=0$, and let $U(\tau,q)$ be the solution of problem \ref{U}.
 Then
the following expansion holds true in a neighborhood of the origin, uniformly with respect to $\tau$,
\begin{equation}
\label{Gexp}  G(\tau,q)=  \sum_{n=1}^\infty q^n G_n(\tau),  \quad \tau \in \R, \end{equation}
where $G_n(\tau)$ is an even $\pi$-periodic trigonometrical polynomial of degree  $2n$ as in formula \ref{exp1}.
\end{lemma}

\begin{proof}
From our assumptions we may write, for $q$ and $x$ sufficiently small,
\begin{equation}
\label{g}
g(x) = \sum_{k=1}^\infty \gamma_k x^k, \qquad \frac{1}{\Omega(q)} = \sum_{n=0}^\infty \kappa_n q^n.
\end{equation}

By composition of analytical functions, we obtain
$$  g(qU(\tau,q)) = \sum_{n=1}^\infty q^n g_n(\tau),   $$
where the coefficients $g_n(\tau)$ are given by the following expressions,
\begin{equation}
\label{gn}
g_n(\tau) = \sum_{k=1}^n \gamma_k \sum_{h_1+\dots+h_k=n} u_{h_1}\cdots u_{h_k}. \end{equation}

From Proposition \ref{prop1}, and by a simple computation,   we get that $g_n(\tau)$ is an even $\pi$-periodic trigonometrical polynomial  whose degree  does not exceed  $2n$.
The assertion follows since $G_n$, owing to   \ref{Gbeta}, \ref{g} is a linear combination of $g_j$, $j\leq n$, that is
\begin{equation}
\label{gG}
G_n(\tau) = \sum_{j=0}^n g_j(\tau)\kappa_{n-j}.
 \end{equation}
  \end{proof}


\subsection{Conclusion of the proof of Theorem \ref{duffhill}}

Let us write the power series  expansion of $\beta^\pm_N(q)$,
\begin{equation}
\label{betan}
\beta^\pm_N(q)=N^2+\sum_{n=1}^\infty B^\pm_{n}(N)q^n,
\end{equation}
where, from \ref{Gbeta},  the coefficients are given by
\begin{equation}
\label{beta} B_n^\pm(N) = \sum_{j=0}^n \Lambda_j^{\pm}(N)\, \Omega_{n-j}.
\end{equation}

Owing to Lemma \ref{L3}, the assumptions of Theorem \ref{hilldirect} are satisfied by the equation \ref{Zold}. It follows that for any eigenvalue number $N$,  the coefficients in the expansion of $\lambda_N^{\pm}$ satisfy
$\Lambda^+_n(N) = \Lambda^-_n(N)$, for $n<N$,
thus  $B_n^+(N) =B_n^-(N)$, for $n<N$ which   proves the assertion.  In particular for the leading term in the expansion (A), we have
$$
C_N = B_N^+(N) -B_N^-(N)  = \Lambda_N^+(N) - \Lambda_N^-(N).   $$ 
\hfill $\square$
\subsection{Additional results}
In certain  cases it is  possible to provide a more precise asymptotic expansion of $L_N(q)$,  as it is shown in the following  Proposition.

\begin{proposition}
\label{Ogrande2}
Let  $K\geq 1$ be the first non-vanishing  power in the expansion \ref{g} of $g(x)$,   that is $g(x) = \gamma_K x^K + O(x^{K+1})$, $\gamma_K \neq0$.
Then, for every $1\leq N\leq K$,  we have
\begin{equation}
\label{Lk} L_N(q)= C_{K,N}\,q^K +O(q^{K+1}),
\end{equation}

In addition, $C_{K,N} \neq 0$    when $N$ and $K$  have the same parity, whereas $C_{K,N}=0$ when $K-N$ is odd.
\end{proposition}

\begin{proof}
If $K>1$, from formula
\ref{gn}, we get $g_n(\tau)\equiv G_n(\tau) \equiv 0$, for $n<K$. Then, owing to  Remark \ref{rem2}, we have    that $\Lambda_n(N)=0$ for $n<K$.  From formula \ref{beta},  it follows that $ B_n^\pm(N)=N^2 \Omega_n$, for $n<K$. This proves that $L_N(q)=O(q^K)$  for $ 0< n < K$.

Let $K\geq 1$.  By using  condition  \ref{incond}, and formula \ref{Bj},  we can compute the coefficient $\Lambda_{K}^\pm(N)$  for $N \leq K$.  This reduces to
 \begin{equation} 
\label{lambdaNK}
\Lambda_{K}^\pm(N) = - \frac{1}{2} \sum_{i=0}^K G_{i,K} \left( z_{N-2i,0}^{\pm} + z_{N+2i,0}^{\pm}  \right) =-G_{0,K}\mp  \frac{1}{2}G_{N,K}.
\end{equation} 
Then we have $C_{K,N} =  \Lambda_{K}^+(N)- \Lambda_{K}^-(N)  = - G_{N,K}$.
From formulas \ref{gn} and  \ref{gG},  we get that
\begin{equation}
\label{Gg}
 G_K(\tau)=g_K(\tau)=\gamma_K (u_1(\tau))^K=\gamma_K (\cos(2\tau))^K.
\end{equation} %
Since   $G_{N,K}$ is the $2N$-th Fourier coefficient of $G_K(\tau)$, we obtain
\begin{equation}
\label{GNK}
G_{N,K}= \frac{2\gamma_K}{\pi}\int_{-\pi/2}^{\pi/2}{\cos(2\tau)^K \,\cos(2N \tau) {\rm d}\tau}.
\end{equation} %

This integral does not vanish if and only if $K$ and $N$ have the same parity, as it follows by  the following formula
$$ 2^{K-1} \cos(2\tau)^K = \cos(2K\tau) + K \cos(2(K-2)\tau) + \binom{K}{2} \cos(2(K-4)\tau) + \cdots \, .$$

In particular, for $K-N=2m$, we get the expression
$  \displaystyle C_{K,N} =- \frac{ \gamma_K }{2^{K-1}} \binom{K}{m} $.

\end{proof}

For example, if $g(x)= \gamma_4 x^4 + O(x^5)$, the second and fourth tongues have order of tangency equal to $4$, in particular they do not collapse to a single line, while the first and third tongues have a contact of   order   at least $5$.   

As an immediate consequence, if  $g'(0)=\gamma_1 \neq 0$,    the first instability tongue  never reduces to a single curve:

\begin{corollary}
For every function $f$ satisfying the assumptions of Theorem  \ref{duffhill}, if $g'(0)\neq 0$, then the first instability tongue of  equation \ref{H2} cannot collapse to a single line, that is $L_1(q)\neq 0$.
\end{corollary}
\begin{remark}
\label{f dispari g pari}
As we mentioned in the introduction, in our discussion of the instability tongues,  we assumed  that  equation \ref{H2}  has the same   period  $T(q)$ as $u(t)$.  As a matter of fact, the period of $g(u(t))$ may  be a fraction of   $T(q)$;   this  occurs  for instance when   $f$ and $g$ are   odd and even functions respectively,   and    the period of $g(u(t))$ is half the period of $u(t)$.     In this case,    the potential function $
2u^2 + \int_0^u f(x) \, dx$ of equation \ref{duff1}  is an even function, thus $u(t+T(q)/2) = - u(t)$ which yields $g( u(t+T(q)/2)) = g(u(t))$.

 It follows that the real eigenvalues of the problem branch out
only for even $N$, or  in other words $L_N(q) \equiv 0$ for odd $N$.
The asymptotic estimate (A) of Theorem 1 is  of course  satisfied    with $C_N=0$ for odd $N$. 
\end{remark}


\section{Shape of the instability tongues}
\label{additional}
\medskip
   The purpose of this section is to characterize the form of instability tongues related to the system \ref{duff1}--\ref{H2} for small $q$. Applications to some significant  cases related to  the theory of suspension bridges are provided in Section \ref{examples}. 

 From the geometrical point of view, we  observe  that the instability tongues  starting  from $ \beta^{\pm}_N(0)=N^2 $ may be  either ``trumpet shaped"    if one of the curves $ \beta = \beta^{\pm}_N (q) $  is  decreasing  and the other increasing, or   ``horn shaped" if   are both increasing or both decreasing. 
For instance, in the case of the Mathieu equation (see also the following Proposition \ref{alpha2alpha3z}) it is well-known that the first two tongues are    trumpet shaped  while the others are  horn shaped  for small values of $q$.  

 The question is relevant for stability analysis at small energies  when we consider the parameter $\beta$ in \ref{H2} as fixed.
 In case of a     trumpet shaped  tongue, the line $ \beta = N^2 $ falls into the instability region (at least for $ q $ small), and the intersection of the tongue with a straight line $\beta = const $ close to $ N^2 $, after a small interval of stability, intercepts  a long interval of instability. 
Viceversa, for  a    horn shaped  tongue, the intersection with  a straight line $\beta = const $ close to $ N^2 $ is at most  a very small   segment. 

In the following proposition, $\alpha$ and $\gamma$ coefficients refer to the power series expansion of $f$ and $g$ respectively. 
\begin{proposition}
\label{alpha2alpha3z}
The asymptotic behavior of the instability tongues, up to second order in  $q$ is the following:

The first tongue is  always   {\bf trumpet shaped} if $\gamma_1 \neq 0$. It has an approximate length $L_1(q) =- \gamma_1( q+\frac{1}{12}\alpha_2 q^2) +o(q^2)$,  as $q \to 0$.

The second tongue   has an approximate length $L_2(q) = (\frac{1}{8}\gamma_1^2-\frac{1}{24} \gamma_1 \alpha_2- \frac{1}{2}\gamma_2  )q^2 +o(q^2)$,  as $q \to 0$. It may be either   {\bf trumpet or   horn shaped},  depending on the parameters.

As for the next tongues, they are generically    {\bf horn shaped},  with the exception of very particular values of the parameters  for which  $B^\pm_j(N) = 0$,  $j < N$.
\end{proposition}

Although it does not geometrically  correspond to a tongue, we may consider also the case $N=0$, when the (even) periodic  eigenvalue  $\beta = \beta^+_0(q)$ forms the right  boundary of an unbounded region of instability.  In this case we have
$$\beta_0^+(q) = \left[ \frac{\gamma_1}{8} (\alpha_2- \gamma_1) -\frac{\gamma_2}{2} \right] q^2 + O(q^3),
$$
thus the  line $\beta=0$    lies or not in the instability region, at least for small values of $q$,  depending on the sign of  $B_2^+(0)= \gamma_1 (\alpha_2- \gamma_1) / 8 - 
\gamma_2/2 $.

\begin{figure}[htbp]\hspace{0.5cm}
\begin{minipage}{5cm}
\begin{center}
\scalebox{0.4}{\includegraphics{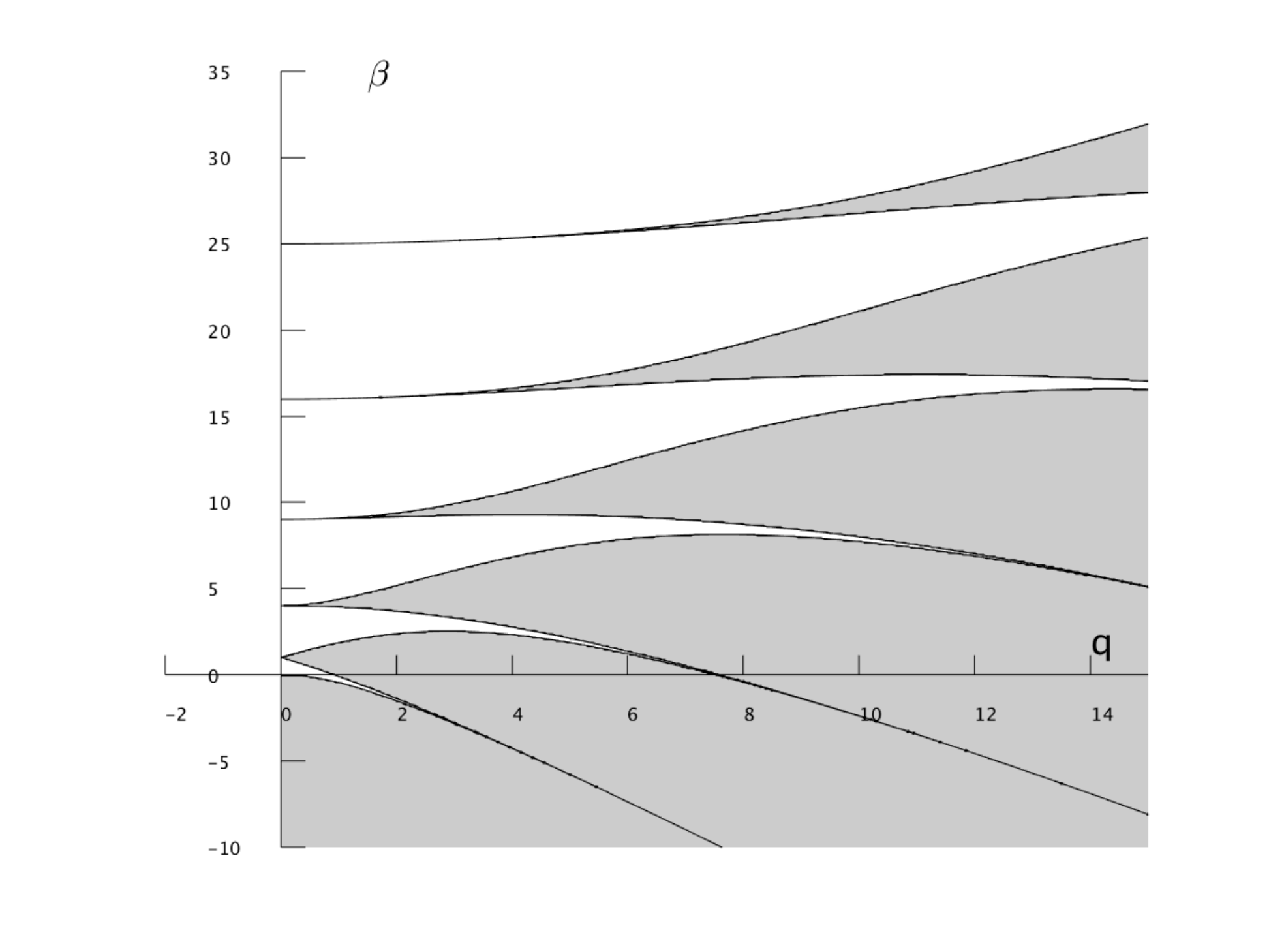}}
\end{center}
\end{minipage}
\caption{Instability tongues of Mathieu equation. The first two tongues are trumpet shaped, the others horn shaped} \label{Fig2}
\end{figure}

The proof of Proposition \ref{alpha2alpha3z} is a consequence of  the following two lemmas.  
 Let us start with direct computation of the first coefficients of $\Omega$, and    $U$ in \ref{Omega}, \ref{U}, in the case when $\alpha_2$,  $\alpha_3$ are not both vanishing,  which  is   the  most interesting  for applications.  
 
\begin{lemma}
\label{alpha2alpha3}
From the first recurrent equations
\ref{uno}, \ref{due}, \ref{tre}, we have the following expressions,
\begin{align}
\label{10} &  \Omega_1=0, \qquad u_2(\tau)=\alpha_2 \left(-\frac{1}{8}+\frac{1}{12}\cos 2\tau+\frac{1}{24}\cos(4\tau)\right),      \\
 \label{20}    &  \Omega_2=-\frac{5}{96} \alpha_2^2+\frac{3}{16} \alpha_3. 
\end{align}

\end{lemma}

\begin{proof} Since  $u_1(\tau)=\cos 2\tau$, equation \ref{due} reads as
$$ u_2'' + 4 u_2'' = 4\Omega_1 \cos 2\tau - \frac{\alpha_2}{2} -  \frac{\alpha_2}{2} \cos 4\tau,$$
thus elimination of the resonant term, and an easy  check yields formula \ref{10}.
Then equation \ref{tre}, after substitution,  becomes
$$   u_3'' +  4u_3= \left(4 \Omega_2 + \frac{5}{24} \alpha_2^2 - \frac{3}{4} \alpha_3 \right)  \cos(2\tau) -  \frac{\alpha_2^2}{12}     - \frac{\alpha_2^2}{12} \cos(4\tau)  - \left(\frac{\alpha_3}{4} + \frac{\alpha_2^2}{24}\right)   \cos(6\tau),  $$
and if one  removes  the resonant term, will get   formula \ref{20}.
\end{proof}

Next from   equation \ref{Zold}, we compute the approximation of the tongues, up to  second power in $ q $. This approximation is significant if $\gamma_1 $, $ \gamma_2 $ are not both vanishing.

\begin{lemma}
The first two coefficients in the expansion \ref{betan} have the following expressions,
\begin{eqnarray*}  B^{\pm}_{1}(1) & =& \mp  \frac{1}{2}\, \gamma_1, \qquad  B^{\pm}_{1}(N)\; =\; 0, \quad \text{ for }  N>1 \; \text{ or } N=0,  \\
 B^{\pm}_{2}(1) &= & \Omega_2+\frac{1}{8}\gamma_1\alpha_2-\frac{1}{2}\gamma_2 -\frac{1}{32}\gamma_1^2  \mp \frac{1}{24} \gamma_1\alpha_2, \\
 B^{\pm}_{2}(2) &=&  4 \Omega_2 + \frac{1}{8}\gamma_1\alpha_2 -\frac{1}{2}\gamma_2 + \frac{1}{24}\gamma_1^2 \mp \frac{1}{48} \left(
\gamma_1\alpha_2 -  3 \gamma_1^2  +  12 \gamma_2  \right), \\
 B^{\pm}_{2}(N) &=&  N^2 \Omega_2 + \frac{1}{8}\gamma_1\alpha_2 -\frac{1}{2}\gamma_2 + \frac{1}{8(N^2-1)}\gamma_1^2,  \quad \text{ for }  N>2 \text{ or } N=0 .
 \end{eqnarray*}

\end{lemma}
\begin{proof}
We go back to \ref{Zold} and observe that the first terms of $G(\tau,q)$ in \ref{gG} are given by
\begin{equation}
\label{g12}
G_1(\tau)=g_1(\tau)=\gamma_1 u_1(\tau),\quad
G_2(\tau)=g_2(\tau)=\gamma_1 u_2(\tau)+\gamma_2 u_1^2(\tau)
\end{equation}
being $\kappa_0=1$, $\kappa_1=0$ in \ref{gG}, and  $g_1(\tau)$,  $g_2(\tau)$ as in \ref{gn}.

Then we insert $u_1(\tau)=\cos(2\tau)$ and $u_2(\tau)$ as in in \ref{g12} of Proposition \ref{alpha2alpha3}, and obtain the following coefficients
\begin{eqnarray*}
G_{0,1}&=&0, \quad G_{1,1} = \gamma_1,\\
G_{0,2}&=&-\frac{1}{8}\gamma_1 \alpha_2+\frac{1}{2}\gamma_2,\quad G_{1,2} = \frac{1}{12}\gamma_1 \alpha_2,
\quad G_{2,2} = \frac{1}{24}\gamma_1 \alpha_2 + \frac{1}{2}\gamma_2.
\end{eqnarray*}

Finally, since  $\Lambda_0^\pm (N)=N^2$, $\Omega_0=1$, $\Omega_1=0$, we have in \ref{beta}
$$B_1^\pm(N)=\Lambda_1^\pm(N), \quad B_2^\pm(N)=\Lambda_2^\pm(N)+N^2 \Omega_2,$$
and by simple substitutions in \ref{B2}, \ref{B22}, \ref{B23}, we have the assertion.
\end{proof}

One may wonder if there exists some universal  upper bound for the number of trumpet shaped tongues. 
In the following proposition we provide a negative answer, by showing that, with a suitable choice of the functions $f$, $g$,    the number of
trumpet shaped tongues can be arbitrarily large.  

\begin{proposition}
\label{trombettine}
Let $K\geq 1 $ be an odd integer, and let   $\alpha_{K+1}$, $\gamma_{K}$ be the first non-vanishing coefficients in the power series expansion of $f$ and $g$ respectively.
Then the tongues corresponding to odd $N$, for $1\leq N\leq K$, are trumpet shaped, and their order of tangency at $q=0$  is exactly $K$.
\end{proposition}
\begin{proof}
For $K=1$ the statement follows from Proposition \ref{alpha2alpha3z}. 
Let us consider $K\geq 3$. We claim that   in the power series \ref{Omega} of $\Omega(q)$,  
we have $\Omega_j=0$ for $1\leq j \leq K$.

 Since $\alpha_j=0$, and  for  $2\leq j\leq  K$, by  a simple inductive argument applied to the recursive equations \ref{n}, we have  that $u_j =0$ for $1\leq j\leq K$, and $\Omega_j =0$ for $1\leq j\leq K-1$. 
 
 It remains to prove that $\Omega_K=0$. 
The equation for $u_{K+1}$ reduces to
$$ u_{K+1}'' +4 u_{K+1}= 4 \Omega_K u_1 -\alpha_{K+1} u_1^{K+1},$$
and the coefficient $\Omega_K$ is computed by removing the resonance term $\cos(2t)$ in the right-hand side term.  Therefore we get
$$4 \Omega_K=\alpha_{K+1}\, \frac{2}{ \pi} \int_{-\pi/2}^{\pi/2}  u_1^{K+1}(\tau) \cos(2\tau) \, {\rm d}\tau = \alpha_{K+1}\, \frac{2}{ \pi} \int_{-\pi/2}^{\pi/2} \cos^{K+2}(2 \tau)\,{\rm d}\tau .$$

The claim is proved, since  this integral   vanishes when $K$  is odd.\footnote{We remark that for even $K$  this last integral is not vanishing, therefore $\Omega_K\neq 0$}

Now, from formula \ref{beta}, it follows that 
 $  B^\pm_{j}(N)=  \Lambda^\pm_j(N)  $, for  $1\leq j \leq K$, $N\leq K$. 
  In addition, since $G(\tau,q) = g(qU(\tau,q))/\Omega(q) = \gamma_K q^K \cos(2\tau)^K + O(q^{K+1})$, owing to Remark \ref{rem2}, we get 
 $$   B^\pm_{j}(N)=  \Lambda^\pm_j(N) =0  \qquad (1\leq j < K).$$ 
  
 On the other hand,  from formula \ref{lambdaNK} in Proposition \ref{Ogrande2}, we have 
 $$B^\pm_{K}(N)  =\Lambda^\pm_K(N) = - G_{0,K}\mp  \frac{1}{2}G_{N,K}, $$ 
 where $G_{N,K}$,  as    computed by formula   \ref{Gg} is not zero, if     $N$ has the same parity of $K$. 
 Finally, for odd $K$, we get 
 $$G_{0,K}= \frac{\gamma_K}{\pi}\int_{-\pi/2}^{\pi/2}{\cos^K(2\tau) \,{\rm d}\tau}=0. $$
 
The conclusion is that   $B_j^{\pm}(N) =0$, for $1\leq j \leq  K-1$, and  
$B^+_{K}(N)= - B^-_{K}(N)\neq 0$, which  proves the assertion.
\end{proof}


\begin{remark}
In many applications the function $g$ is proportional to the derivative of $f$, i.e. $g(x)=\tilde\gamma f'(x)$. In these cases we obviously have
$ \gamma_n=0\, \Longleftrightarrow  \,\alpha_{n+1}=0 $

 Under this assumption, Proposition \ref{trombettine} yields examples of trumpet shaped tongues with the same order of tangency.
\end{remark}

 
\section{Applications to suspension bridges and examples}
\label{examples}

In this section we come back to the problem that gave rise to  our investigations, and we illustrate a few results related to problem (II) (see introduction). 

An important issue in the mathematical modeling of suspension bridges is the phenomenon  of energy transfer from flexural to torsional
modes of vibration along the deck of the bridge. According to a recent field of research   \cite{AG,BG,Gaz,F,CAG} internal nonlinear resonances   giving rise to 
 the onset of instability may occur even when the aeroelastic coupling is disregarded. In particular, in the fish-bone bridge model  (\cite[ch. 3]{Gaz},  or   \cite{MP}),  the non-linear coupling between flexural and torsional oscillation of the bridge is described by  the function 
 $  \mathcal{F}(x)$, which represents   in the PDEs system the restoring action of the pre-stressed hangers.\footnote{In the cited works $\mathcal{F}$ is written as $f$; we changed the font  to avoid confusion}  
A first expression of such $  \mathcal{F}$  was proposed in  \cite{MCKW,Moo}:
$$
  \mathcal{F}(x) = {\rm k} \left[ (x+x_0)^+ - x_0 \right].  
$$  

 Under this assumption, the  PDEs system acts as a linear  uncoupled system for sufficiently low  energy.

Anyway, other expression of  $  \mathcal{F}$ have been proposed in \cite{MP,BFG,MP2} and some of these are nonlinear and analytical  function in a neighborhood of the origin. In that case some instability zone for low energy may be expected.

The second step in the cited papers  is to reduce the PDE-system to an ODEs one, through a Galerkin projection. If, for sake of simplicity, our aim is to study the interaction between a single torsional mode and a single flexural one (the first ones, for example),
the instability at a given energy level of a pure flexural solution is equivalent to the instability of an Hill equation like \ref{H2}.
More precisely, we are led to study a system of  two coupled equations (the linearized system around the pure flexural solution). Such ODEs system can be written in the form \ref{H2}--\ref{duff1}\footnote{ The coefficient 4 in
\ref{duff1} can always be fixed with a suitable rescaling in time.} where the function $f(x)$ in \ref{duff1}   is strictly related to the function $  \mathcal{F}$ in the PDEs model and the functions $g$ and $f$ in  \ref{H2}--\ref{duff1} satisfy $g(x)=\tilde \gamma f'(x)$, $\tilde \gamma>0$, (see \cite{BG,MP}).

Our work proves that the thickness of the instability tongues gets thinner and thinner for growing $N$, then the most significant instability zones correspond to the first tongues; 
moreover,  the parameter $\beta$ being  constant in the applications,
 the shape of the tongues is also important, because  entering deeply an instability zone  is more destructive than being near to its border.
 
Now we present some simple examples of application of Proposition \ref{alpha2alpha3z}. 

\
 
\noindent {\bf Example 1.}
Our first example is given by the following system,
 \begin{align*}
& u''(t,q)  +   4 u(t,q) + \alpha u^2(t,q)=0 ,  \qquad u(0;q)=q, \quad u'(0;q) =0,   \\
 & z''(t)      +     (\beta +2 \tilde\gamma \alpha u(t,q)) z(t) =0.
 \end{align*}

Owing to  Propositions  \ref{alpha2alpha3}, we know that the first tongue  is  trumpet shaped and length $L_1(q)=- 2\tilde\gamma \alpha q + O(q^2)$.
The second tongue is trumpet shaped if and only if
$$\tilde \gamma<-1,\quad \frac{1}{2}<\tilde \gamma<1, \quad \tilde\gamma>\frac{5}{2}.$$

We can also prove that    coexistence may occur for special  values of the parameters; precisely 
if  $ \tilde\gamma=\frac{n(n+1)}{12}$ ($n\in \N$) , then there exist only $n$ instability tongues, or equivalently there exist  $2n+1$ simple eigenvalues.

In fact, if   we set $\gamma=2 \tilde\gamma \alpha$ for sake of simplicity, and
plug $Q(t)=\gamma u(t)$ into \ref{Q}, we get
 $$u''+Au+B/\gamma+3\gamma u^2=0, $$
 which is satisfied with the choice  $A=4$, $B=0$, $\gamma=\alpha/3$. Thus the result follows by Theorem \ref{incemod}.

The following formula (see  \cite[Th. 5.3]{Vol}) shows that the simple  eigenvalues  are the lowest ones:
$$
C_N  = \frac{(-1)^N\, \alpha^N}{8^{N-1} \, ((N-1)!)^2 } \prod_{k=0}^{N-1} \left(2\tilde \gamma  - \frac{k(k+1)}{6} \right).  
$$  
 
In addition $C_N \neq 0$ for every $N$, if $\tilde \gamma$ does not take one of the values $n(n+1)/12 $.

\

\noindent {\bf Example 2.}
Our second example has been discussed  for fixed values of the parameter $\tilde\gamma$ in  \cite{GG} ($\tilde\gamma=1/3$), and  \cite{BG} ($\tilde\gamma=3$).  It is provided by the following coupled system,
 \begin{align*}
& u''(t,q)  +   4 u(t,q) +\alpha  u^3(t,q)=0 ,  \qquad u(0;q)=q, \quad u'(0;q) =0,   \\
 & z''(t)      +     (\beta +3\tilde \gamma \alpha u^2(t,q)) z(t) =0.
 \end{align*}

We observe that this second example falls within the conditions of Remark \ref{f dispari g pari}, so that the coefficient $g(u)$  has  fundamental  period $T(q) / 2$.
Thus   the genuine instability tongues    branch off from the $\beta$-axis at  $\beta_N(0) = (2N)^2$, $N\in \N$.

 The first tongue is trumpet shaped if and only if
$$\frac{1}{3}<\tilde\gamma<1.$$

Coexistence may occur for some values of the parameters; precisely 
if  $\tilde\gamma=\frac{n(n+1)}{6}$, then there exist only $n$ instability tongues (in particular if $\tilde\gamma=\frac{1}{3}$, there is only the first one).

To prove this last assertion, let us set $\gamma=3\tilde \gamma \alpha$ and
   $Q(t)=\gamma u^2(t)$,  and plug it into \ref{Q}. We obtain
 \begin{equation}
  \label{Qu}  (u')^2+u u''+\frac{A}{2}u^2+\frac{B}{2 \gamma}+\frac{3}{2}\gamma u^4=0.
   \end{equation}

The first equation multiplied by $u'$ yields the identity,
$$(u')^2+4u^2+\frac{\alpha}{2}u^4=2 E(q),$$
where  $E(q)=4q^2+ \alpha q^4/2$ is the energy of  $u$. By replacing $(u')^2$ in \ref{Qu}, we get
$$u u''+(\frac{A}{2}-4)u^2+(\frac{3}{2}\gamma-\frac{\alpha}{2})u^4+(2 E-\frac{B}{2 \gamma})=0.$$

Choosing $B=4 \gamma E$ we get rid of the constant term. Finally by setting  $A=16$, $\gamma=\alpha$, equation  \ref{Q} is satisfied.

\

 \noindent {\bf Example 3.}
In \cite{MP2} we  numerically studied the behavior of the ODEs system for some other functions. 
One of those was
$$
\tilde f(x)=mx +m\sqrt{x^2+(h/m)^2}\, \,-h =mx+\frac {m^2}{2 h} x^2+ O(x^4),
$$ 
where $m$, $h$, are positive constants. The corresponding non linear perturbations $f$ and $g$ in the linearized  system  \ref{H2}--\ref{duff1} become, after the rescaling:
$$ f(x)=\alpha x^2+O(x^4),\quad g(x)=2 \tilde \gamma \alpha x+O(x^3),$$
where $\alpha$ is a suitable positive constant.

The asymptotic behavior of the first two tongues for this choice of non-linearity is identical to the one of the first example.
Besides we have no information about the coexistence.

Looking at these examples, we can note that the role of the parameter $\tilde \gamma$ which depends on the structural constants in the PDEs model, is the most relevant for the shape of the first tongues.

Our last example about coexistence is inspired by the examples 1 and 2 and appears to be novel.

\

\noindent {\bf Example  4.}
  Let us consider the following coupled system
 \begin{align*}
& u''(t)  +   4 u(t) + f(u(t))=0 , \qquad u(0)=q, \quad u'(0) =0 , \\
 & z''(t)      +     (\beta +g(u(t))) z(t) =0,
 \end{align*}
with $f(x)= \alpha_2 x^2+\alpha_3x^3$, $g(x)=\gamma_1 x+\gamma_2 x^2$.

This system has exactly  $2n+1$ simple eigenvalues  (the first ones)
if $f$ and $g$ satisfy the following conditions:
$$f(x)= \alpha x^2+\frac{\alpha^2}{18}x^3, \quad g(x)=\frac{n(n+1)}{6} f'(x)\quad  \alpha \in \R,\, \alpha\neq 0,\, n\in \N.$$

The verification is cumbersome  but   follows the lines of the two first examples.  

\appendix

\section{Recursive formulas for the computation of $C_N$}

Our goal here is to provide a recursive formula for the computation of the leading coefficient $C_N$ in the asymptotics of $L_N(q)$.

\begin{proposition} Let us consider equation \ref{Hill1}   when $G(t,q)$ is given by   \ref{exp}--\ref{exp1}. For $0\leq p \leq N$, let the numbers $r_p(N)$ be recursively defined  by the  rule,
\begin{equation}
\label{deltaz}
r_p(N)  = - \frac{1}{8p(N-p)}\sum_{s=1}^p G_{s,s}\, r_{p-s}(N), \qquad r_0(N)=2.
\end{equation}

Then   the following formula holds true,
\begin{equation}
\label{deltaB} \Lambda_N(N)^+ -\Lambda_N(N)^-  =   - \frac{1}{2} \sum_{p=0}^{N-1} G_{N-p,N-p}\, r_{p} (N) .
\end{equation}
\end{proposition}
\begin{proof}

Let us set  $\Delta z_{k,n} = z_{k,n}^+ - z_{k,n}^-$, where $ z_{k,n}^{\pm}$ are defined by \ref{fourier}. Owing to formula \ref{Bj} for $n=N$, we have
$$
\Lambda_N(N)^+ -\Lambda_N(N)^-  =   - \frac{1}{2} \sum_{s=1}^N \sum_{i=0}^s G_{i,s} \left( \Delta z_{N-2i,N-s} + \Delta z_{N+2i,N-s} \right).
$$

Thanks to Lemma \ref{L3}, the only non-vanishing terms of the right-hand side are those having index along the line $k=2n-N$ (we refer to the notations of Lemma \ref{L3}), that is $\Delta z_{N-2i,N-s}$ for $i = s$.
Therefore we get
$$ \Lambda_N(N)^+ -\Lambda_N(N)^-  =   - \frac{1}{2} \sum_{s=1}^N G_{s,s}\, \Delta z_{N-2s,N-s} .$$

By using the notation   $r_{N-s} (N) = \Delta z_{N-2s,N-s}$, and by inverting the order of summation, we get \ref{deltaB}.

As for the formula \ref{deltaz}, 
we note that $r_p(N) = \Delta z_{-N+2p,p}$, and that   the pair  $(-N+2p,p)$ lies on  the line  $k=2n-N$.
Owing to formula \ref{ric} with $k=-N+2p$, $n=N$,   with analogous considerations we get,
\begin{eqnarray*}
4p(N-p)  r_p(N) &=& 4p(N-p) \Delta z_{-N+2p,p} = - \frac{1}{2} \sum_{s=1}^p G_{s,s}\, \Delta z_{-N+2p-2s,p-s}   \\
&=& - \frac{1}{2} \sum_{s=1}^p G_{s,s}\, r_{p-s}(N).
\end{eqnarray*}

This proves the assertion since, thanks to \ref{incond}, $r_0(N)=  \Delta z_{-N,0} = z_{-N,0}^+ - z_{-N,0}^-  =2$. \end{proof}

\begin{remark}   It is clear from \ref{deltaz}--\ref{deltaB} that $ \Lambda_N(N)^+ -\Lambda_N(N)^- $ is a polynomial 
of degree $N$ in the diagonal coefficients 
$G_{j,j}$, $ 1\leq j \leq N$. 
It is not difficult (but cumbersome) to show that it takes the form
 \begin{equation}
 \label{GN}
 - G_{N,N} + P_N(G_{1,1}, \dots, G_{N-1,N-1}),
 \end{equation}
 where $P_N$ is a linear combination of 
$$ \prod_{j=1}^{N-1} G_{j,j}^{p_j} \qquad \text{with} \quad  \sum_{j=1}^{N-1}  j p_j = N.   $$

In particular, 
the   monomial of degree $N$ is given by 
$$   \frac{(-1)^N}{((N-1)!)^2\, 8^{N-1}}\,G_{1,1}^N , $$
in accordance with the known asymptotic expansion of the Mathieu equation \cite{LK}. 
\end{remark} 

Let us now consider equation \ref{H2}.    In order to compute $G(\tau,q)= g(qU(\tau,q))/\Omega(q)$, we have to go back to Section 3, and look at the expansion \ref{Gexp}, whose coefficients are given by    \ref{gn}--\ref{gG}. 

We need a notation: 
given any trigonometrical polynomial $F(\tau)$,  let $P_{2n}[F]$ be its  $\cos(2n \tau)$-coefficient, i.e $P_{2n}[F] = 1/\pi \, \int_{-\pi}^\pi F(\tau) \cos (2n\tau) {\rm d}\tau$.
Owing to formula \ref{gG} (recall that $\kappa_{0}=1$) we have that 
$$G_{n,n} = P_{2n}[G_n] = P_{2n}[g_n].$$  

\begin{proposition}
Under the assumptions of Theorem \ref{duffhill}, let us consider the expansion \ref{expansion} in Section 3. Let  us set
$ \displaystyle  A_n = \frac{1}{2}P_{2n}[u_{n}]  $ ($n\geq1$),
and define the generating functions,
$$ \psi(q) = \sum_{n=1}^\infty A_n q^n, \qquad \Psi(q) = \frac{1}{2} \sum_{n=1}^\infty G_{n,n} q^n.$$

Then $\psi(q)$ solves the differential equation
\begin{equation}
\label{genode} q^2 \psi''(q) + q\psi'(q) - \psi(q) = \frac{1}{4} f(\psi(q)),
\end{equation}
with the initial conditions $\psi(0)=0$, $\psi'(0) = \frac{1}{2}$. In addition, we have
\begin{equation}
\label{Grec} \Psi(q) = g(\psi(q)).
\end{equation}
\end{proposition}

The introduction of the generating functions is just for compactness of notations.
The differential equation \ref{genode}, and formula \ref{Grec} are equivalent to the following recursive formulas:
\begin{equation}
\label{Arec}
A_1=\frac{1}{2}, \qquad 4(n^2 -1)A_n  = \sum_{m=2}^n \alpha_m \sum_{h_1+\dots+h_m=n} A_{h_1}\cdots   A_{h_m} \qquad (n \geq 2),
\end{equation}
\begin{equation}
\label{Grecrec}
\frac{1}{2} G_{n,n} = \sum_{m=1}^n \gamma_m \sum_{h_1+\dots+h_m=n} A_{i_1}   \cdots  A_{i_m} .
\end{equation}

\begin{proof}
Let us set $\zeta = e^{2i\tau}$.
By definition of $A_n$, we have
$$ u_n = A_n (\zeta^n + \zeta^{-n}) +  \rm{l.o.t.} $$
where by $\rm{l.o.t.}$ we denote powers of $\zeta$ with modulus less than $n$. By plugging this expansion into the recursive  equation \ref{n}, we get
\begin{eqnarray*}
   4(  n^2 -1)A_n (\zeta^n + \zeta^{-n})  & = &    \sum_{k=2}^{n} \alpha_{k} \sum_{i_1 +\dots+ i_{k}=n}A_{i_1} (\zeta^{h_1} + \zeta^{-h_1}) \cdots   A_{i_1} (\zeta^{h_k} + \zeta^{-h_k}) + \rm{l.o.t.} \\
 & & =  \sum_{k=2}^{n} \alpha_{k} \sum_{i_1 +\dots+ i_{k}=n}A_{i_1}   \cdots  A_{i_k} (\zeta^n + \zeta^{-n}) + \rm{l.o.t.} \end{eqnarray*}

 Neglecting the $\rm{l.o.t.}$, we obtain formula \ref{Arec} for $n\geq 2$.
 Multiplying \ref{Arec} by $q^n$ and summing up, we obtain formula \ref{genode} since
 $$ \sum_{n=2}^\infty  ( n^2 -1)A_nq^n = q^2\psi''(q) + q\psi'(q) - \psi(q) .$$

 Let us now consider the coefficient $G_{n,n} =P_{2n}[g_n]$, where $g_n$ is given by formula \ref{gn}. Proceeding as before, we have
$$ \frac{1}{2} G_{n,n} (\zeta^n + \zeta^{-n})= \sum_{m=1}^n \gamma_m \sum_{h_1+\dots+h_m=n} A_{i_1}   \cdots  A_{i_m} (\zeta^n + \zeta^{-n}) + \rm{l.o.t.} $$
which yields formula \ref{Grecrec}
\end{proof}

In the simplest  non-trivial example,   $f(x)= \alpha x^2$, $g(x)= x$, we have
\begin{equation}
\label{Glame} G_{n,n} =  \frac{n}{8^{n-1}} \, \left(\frac{\alpha}{6}\right)^{n-1},
\end{equation}
as we may directly verify from \ref{Arec}--\ref{Grecrec} which reduce to
$$G_1=1, \qquad  (n^2 -1)G_{n,n} = \frac{\alpha}{8}\,   \sum_{j=1}^n G_{j,j}G_{n-j,n-j} \qquad (n \geq 2). $$

In fact upon substitution \ref{Glame}, and simplification, we obtain the well-known identity, 
$$ \frac{(n^2-1)n}{6}  =  \sum_{j=1}^n j(n-j) . $$


\section{The forms of  the Ince theorem}

 We think that it could be useful for the reader to have some general information about the classical  Lam\'e equation and the Ince theorem. 
First of all  the Lam\'e equation has five different forms, and this  can be a bit confusing: we have the ``Jacobian" form and the ``Weierstrassian" form, 
that are   Hill equations, 
 two algebraic forms, and the trigonometric form which is  of Ince's type. Here we present the first two versions.  
 
The Jacobian form is given by the following equation, 
\begin{equation}
\label{LJ}
 y''(x) + (\lambda-n(n+1)\, k^2\,\textrm{sn}^2(x))y(x)=0,
\end{equation}
where $\textrm{sn}(x)$ is  the  Jacoby elliptic sine function of modulus $k^2$, and  $n\in \R$ (see e.g. \cite[§ 7.3]{MW}).

The  Weierstrassian form is 
$$
w''(z)+(\beta-n(n+1) \wp(z))w(z)=0  \quad ( z\in \mathbb{C}), 
$$ 
 where 
the Weierstrass function $ \wp(z)= \wp(z; g_2,g_3)$  has  a double pole in $ z=0$, and solves  the following differential equation,  \begin{equation}
\label{P}
(P')^2=4P^3-g_2P-g_3=4(P-e_1)(P-e_2)(P-e_3). \end{equation} %

 Under the assumption that both the invariant $g_2$, $g_3$ and the roots $e_i$ are    real, with $e_3<e_2<e_1$,  $ \wp(z)$ has two semi-periods: $\omega=\omega_1$ which is real, and  $\omega'=\omega_3$, which is pure imaginary (another symbolism that emphasizes the periods is $ \wp(z)= \wp(z| \omega,\, \omega')$).
A complete description of elliptic functions and their properties can be found in \cite{AS1,WW}. 

Anyway, if we are interested only in real solution of \ref{P}, its general integral is given by $\wp(t+\omega_3+c)$, where $\omega_3\in i\R$, $c\in \R$, and the Weierstrassian form of the Hill equation becomes,  
\begin{equation}
\label{LWR}
w''(t)+(\beta-n(n+1) \wp(t+\omega_3))w(t)=0 \quad (t\in \R).
\end{equation} %

 In \cite[ch. XXII, § 23.4]{WW} (also the formulas in  \cite[§ 18.9]{AS1} can be helpful) we can find how to transform  equation  \ref{LWR} into \ref{LJ}. 
  The simplest identity that shows the connection between the two forms  is, 
\begin{equation*}
 \wp(t+\omega_3)=e_3+(e_2-e_3)\textrm{sn}^2(\sqrt{e_1-e_3}\,t); 
 \end{equation*} 
then, with the rescaling $x=\sqrt{e_1-e_3}\,t$, it is easy to pass from \ref{LWR} to  \ref{LJ}, being $k^2=\frac{e_2-e_3}{e_1-e_3}$ exactly the modulus of $\textrm{sn}(x)$. 

The classical Ince theorem, with the Lam\'e equation in Jacobian form, is presented in \cite{MW} and its proof uses the equivalence between the Jacobian and trigonometrical forms of this equation 
(we can find also the substitutions that transform a form into another one, with the exception of  the Weierstrassian form,  in \cite[§9.1]{Ar}). 
The alternative version of the Ince theorem in Weierstrassian form is widely cited (see for example \cite{GW} ) and has its merits:

 \begin{theorem}
\label{InceW}
Let $\wp(t)  = \wp(t|\omega_1,\omega_3)$ be the elliptic Weierstrass function  with periods $\omega_1\in \R$, $\omega_3\in i\R$, and let
\begin{equation}
\label{ince}
\tilde{Q}(t)  = -n(n+1) \wp(t+\omega_3 +c), \qquad c\in \R .
\end{equation}
be the Lam\'e--Ince potentials.

Then, for every positive integer $n$, the Hill equation
 $$ w'' + (\lambda +  \tilde{Q}) w=0$$
has   exactly      $n+1$ instability intervals, including the unbounded one.
\end{theorem}

Now we show that Theorem \ref{incemod} in Section  \ref{Mathieu} is no more than a simple consequence of  Theorem \ref{InceW}, which means that for $n=1$ the necessary and sufficient condition \ref{Q} and the  Ince  theorem are equivalent. This is no longer true for $n>1$, where a Lam\'e--Ince potential satisfies all the KdV equations of order $k\geq n$, but it is well known that such potentials, for $n>1$, don't describe all the solutions of the KdV hierarchy. 

Again we point out    that this is not a new result (see \cite[Th. 7.13]{MW}, where it is presented without proof). 
\vskip .2truecm 
\noindent {\em Proof of Theorem \ref{incemod}. }
Let $Q$ be a periodic not constant solution  of \ref{Q}, then it   also solves the following equation,  
\begin{equation*}
(Q')^2+2Q^3+AQ^3+2BQ=2E,
\end{equation*}
with
$\frac{A^2}{12}-B>0$
and  $E$ such that the roots of the equation
$$ 2Q^3+AQ^2+2BQ-2E=2(Q-Q_1)(Q-Q_2)(Q-Q_3)=0$$
are real distinct numbers.
Operating the following substitution
$$Q=-2P-\frac{1}{6}A,$$
 we obtain that $P$ satisfies \ref{P}. Then  we have  $Q(t)= -\frac{A}{6}-2 \wp(t+\omega_3+c)$, for a suitable $c \in \R$.

Then the Hill equation
 $$ z'' + (\beta +  Q(t)) z=0$$
becomes
 $$ z'' + (\beta -\frac{A}{6}-2 \wp(t+\omega_3+c) )z=0,$$
that satisfies the Ince Theorem for $n=1$, with $\lambda=\beta -\frac{A}{6}$, $\tilde Q=-2 \wp$.

Let us define  $Q_n=\frac{n(n+1)}{2}Q(t)$, with  $Q(t)$ satisfying \ref{Q}. Then
$$Q_n=-\frac{n(n+1)A}{12}-n(n+1) \wp(t+\omega_3+c)$$
satisfies the  hypotheses of  the Ince theorem for every positive integer $n$, bar a translation, absorbed by the eigenvalue  $\lambda$.
\qed

\section*{Acknowledgements}
We wish to thank FILIPPO GAZZOLA for valuable suggestions and comments.


\end{document}